\documentclass[authoryear,12pt]{elsarticle}
\usepackage{mathrsfs}
\usepackage{amsthm}
\usepackage{amssymb}
\usepackage{amsmath}
\usepackage{amsfonts}
\usepackage{amssymb}
\usepackage{graphicx}
\usepackage{color}%
\usepackage{hyperref}
\biboptions{sort&compress}
\newtheorem{theorem}{Theorem}%  meant for continuous numbers
\newtheorem{proposition}[theorem]{Proposition}%
\newtheorem{remark}{Remark}%
\newtheorem{definition}{Definition}%
\newtheorem{corollary}{Corollary}

\def\lt{\left}
\def\rt{\right}
\def\d{d}
\def\vp{\varphi}

\def\be{\mathbb{E}}

\def\bn{\mathbb{N}}

\def\cp{\mathcal{P}}
\def\cq{\mathcal{Q}}
\def\cf{\mathcal{F}}
\def\cb{\mathcal{B}}
\def\cm{\mathcal{M}}
\def\cf{\mathcal{F}}

\def\cx{\mathcal{X}}

\def\be{{\mathbb{E}}}

\begin{document}
\begin{frontmatter}
\title{Generalized Divergence Measures and Weak Convergence for the Sets of Probability Measures\footnote{This work is supported by NSF of China (No.12326603 and No.11601281) and the Qilu Young Scholars Program of Shandong University.}}
\author[1,2]{Xinpeng Li}
%\ead{lixinpeng@sdu.edu.cn}
\author[1]{Miao Yu}
\ead{yu-miao@mail.sdu.edu.cn}
\address[1]{Research Center for Mathematics and Interdisciplinary Sciences, \\Shandong University, 266237, Qingdao, China}
\address[2]{Frontiers Science Center for Nonlinear Expectations (Ministry of Education), \\Shandong University, 266237, Qingdao, China}
\begin{abstract}
This paper extends the asymmetric Kullback–Leibler divergence and symmetric Jensen–Shannon divergence from two probability measures to the case of two sets of probability measures. We establish some fundamental properties of these generalized divergences, including a duality formula and a Pinsker-type inequality. Furthermore, convergence results are derived for both the generalized asymmetric and symmetric divergences, as well as for weak convergence under sublinear expectations.
\end{abstract}
\end{frontmatter}

\section{Introduction}
In classical probability theory, for two continuous random variables with probability densities $p(x)$ and $q(x)$, the Kullback-Leibler divergence (KL divergence) from $p(x)$ to $q(x)$ is defined as follows:
$$\text{KL}(p\parallel q)=\int p(x)\text{log}\lt(\frac{p(x)}{q(x)}\rt)\d x,$$
where $\text{log}$ denotes the natural logarithm. The definition for discrete random variables is analogous. A key characteristic of the KL divergence is its asymmetry, i.e., $\text{KL}(p\parallel q) \neq \text{KL}(q \parallel p)$ in general. To address this, \cite{lin1991divergence} introduced a symmetric divergence known as the Jensen-Shannon divergence (JS divergence), which is defined in the following form: 
$$\text{JS}(p,q)=\frac{1}{2}\text{KL}\lt(p\parallel\frac{p+q}{2}\rt)+\frac{1}{2}\text{KL}\lt(q\parallel\frac{p+q}{2}\rt).$$

As fundamental measures of dissimilarity between probability distributions, the KL and JS divergences possess a range of important theoretical properties(see \cite{lin1991divergence,cover2006elements}) and are widely used in applications such as variational inference (see \cite{blei2017variational}) and model selection (see \cite{schennach2017simple}). However, these applications typically rely on divergences between two probability distributions, where one distribution is precisely specified. In practice, random variables of interest often reside in uncertain environments such that their distributions belong to sets of probability measures. This naturally raises a fundamental question: how can we quantify the dissimilarity between two such sets? To address this question, we first define a generalized KL divergence using the $\sup\inf$ operator and then define a generalized JS divergence via KL divergence. Furthermore, our analysis extends to weak convergence under sublinear expectations—a framework pioneered by Peng (see \cite{peng2019}), which provides a powerful tool for handling such distributional uncertainty. Under this framework, a random variable $X$ is associated with a family of distributions $\{F_X(\theta,\cdot)\}_{\theta\in\Theta}$, reflecting its distributional uncertainty. These generalized divergences allow us to quantify the discrepancy between two such families corresponding to random variables $X$ and $Y$. Moreover, in scenarios where prior information is limited, sublinear expectation theory offers a principled approach, as evidenced by its growing use in data analysis (see \cite{JPY}) and financial risk management (see \cite{pei2021worst, PYY}).

The remainder of this paper is structured as follows. Section 2 reviews the definitions and key properties of the KL divergence, JS divergence, and total variation distance, and then extends these concepts to the case of sets of probability measures. Section 3 establishes a duality formula for the generalized KL divergence and derives fundamental properties of these generalized divergences. Section 4 presents a generalized Pinsker’s inequality and compares the strength of convergence under various metrics, including weak convergence in the context of sublinear expectations.

\section{Preliminaries}
The KL divergence, first introduced by \cite{kullback1951information}, is formally defined on a measurable space $(\cx, \cf)$.
% The KL divergence was first introduced by \cite{kullback1951information} from a statistical 
% perspective, defined in a symmetric form. However, subsequent research developments often 
% considered asymmetric forms. This paper also examines the asymmetric form.  Let $\mu$ and $\nu$ %be two probability measures on the fixed measurable space $(\cx,\cf)$.

\begin{definition}\label{CKL}
    For two probability measures $\mu$ and $\nu$ on $(\cx,\cf)$, the KL divergence is defined as follows:
\begin{align*}
    \normalfont
    \text{KL}(\mu\parallel\nu)
    &= \begin{cases} 
    \displaystyle\int_\cx \log\left(\frac{\d\mu}{\d\nu}(x)\right)\mu(\d x) 
    &\quad \text{if } \mu\ll\nu \\[6pt]
    +\infty 
    &\quad \text{otherwise}
    \end{cases}
\end{align*}
where $\mu\ll\nu$ means that $\mu$ is absolutely continuous with respect to $\nu$.
\end{definition}
We note that $\textup{KL}(\mu\parallel\nu)\geq0$ and $\textup{KL}(\mu\parallel\nu)=0$  if and only if $\mu=\nu$. Furthermore, $\textup{KL}(\cdot,\cdot)$ is a convex function (see  \cite{van2014renyi}).   The following Pinsker's inequality and duality formulas can be found in \cite{boucheron2013} and \cite{Donsker1983AsymptoticEO}.

\begin{proposition}[Pinsker's inequality]\label{pinsker}
    %\normalfont\text{(Pinsker's inequality)}
    For two probability measures $\mu$ and $\nu$ on $(\cx,\cf)$, we have        
    $$\textup{KL}(\mu\parallel\nu)\geq 2V^2(\mu,\nu),$$
    where $V(\mu,\nu)$ is the total variation distance between two probability measures $\mu$ and $\nu$, i.e.,     
    $V(\mu,\nu)=\sup_{A\in\cf}\vert\mu(A)-\nu(A)\vert.$ 
\end{proposition}
  
\begin{proposition}[Duality formulas]\label{dual}
    %\normalfont\text{(Duality formula for KL-divergence)}
    For two probability measures $\mu$ and $\nu$ on $(\cx,\cf)$, we have  
    \begin{itemize}
        \item[(i)] $\textup{KL}(\mu\parallel\nu)=\sup_{\vp\in L^\infty(\cx)}\lt\{E_\mu[\vp]-\textup{log}\lt(E_\nu[e^\vp]\rt)\rt\},$
        \item[(ii)] $V(\mu,\nu)=\frac{1}{2}\sup_{\vp\in L^\infty(\cx),\Vert \vp\Vert_\infty\leq 1}\vert E_\mu[\vp]-E_{\nu}[\vp]\vert.$
    \end{itemize}
where $L^\infty(\cx)$ denotes the space of all bounded $\cf$-measurable functions, $E_\mu$ and $E_\nu$ are the expectations introduced by $\mu$ and $\nu$, $\Vert\vp\Vert_\infty:=\sup_{x\in\cx}\vert\vp(x)\vert$.
\end{proposition}

\begin{remark}\label{dualremark}
If $\cx$ is a Polish space with the Borel $\sigma$-algebra $\cb(\cx)$, then replacing $L^\infty(\cx)$ by $C_b(\cx)$ gives the same supremum in Proposition \ref{dual} since measurable functions can be approximated by bounded continuous functions.
\end{remark}
By duality formulas, we give a simple proof for the lower semi-continuity of the KL divergence and the total variation distance, which is different from \cite{van2014renyi}.

\begin{proposition}
     If $\cx$ is a Polish space, $\textup{KL}(\mu,\nu)$ and $V(\mu,\nu)$ are lower semi-continuous functions of the pair $(\mu,\nu)$ with respect to the weak convergence.   
\end{proposition}

\begin{proof}
    Let $\{\mu_n\}_{n\geq 1}$ and $\{\nu_n\}_{n\geq 1}$ be sequences of probability measures that weakly converge to $\mu$ and $\nu$, respectively. For each $\vp\in\mathcal{C}_b(\cx)$, we have 
    \begin{align*}
        {\lim\inf}_{n\to\infty}\textup{KL}(\mu_n\parallel\nu_n)&\geq{\lim\inf}_{n\to\infty}\{E_{\mu_n}[\vp]-\textup{log}E_{\nu_n}[e^\vp]\}\\
        &=E_\mu[\vp]-\textup{log}E_\nu[\vp],
    \end{align*}
    then we have $\lim\inf_{n\to\infty}\textup{KL}(\mu_n\parallel\nu_n)\geq\textup{KL}(\mu\parallel\nu)$. The case of $V(\mu,\nu)$ can be proved similarly.  
\end{proof}
 
The notion of KL divergence is asymmetric and the probability measures involved in KL divergence should satisfy the condition of absolute continuity if the divergence is finite. To overcome the previous restrictions, \cite{lin1991divergence} proposed a symmetric divergence, called the Jensen-Shannon divergence. 

\begin{definition}
    For two probability measures $\mu$ and $\nu$ on $(\cx,\cf)$, the JS divergence is defined as follows:
    \normalfont  
    $$\text{JS}(\mu,\nu)=\frac{1}{2}\text{KL}\lt(\mu\parallel\frac{\mu+\nu}{2}\rt)+\frac{1}{2}\text{KL}\lt(\nu\parallel\frac{\mu+\nu}{2}\rt).$$
\end{definition}
We note that this definition does not require absolute continuity, since $\mu,\nu\ll\frac{\mu+\nu}{2}$ is always true. It can be verified that  $0\leq\textup{JS}(\mu,\nu)\leq \textup{log}2$.

% As a bridge connecting the strength of convergence of probability measures under different metrics, the total variation distance is half the $L_1$-distance (see \cite{boucheron2013}), we can immediately derive the dual formula for the total variation distance. 

% \begin{proposition}[Duality formula for total variation distance]\label{dual-v}    
%     %\normalfont\text{(Duality formula for total variation distance)}    
%     For two probability measures $\mu$ and $\nu$ on $(\cx,\cf)$, we have       
%     $$V(\mu,\nu)=\frac{1}{2}\sup_{\Vert f\Vert_\infty\leq 1}\vert E_\mu[f]-E_{\nu}[f]\vert=\frac{1}{2}\sup_{\Vert f\Vert_\infty\leq 1}(E_\mu[f]-E_{\nu}[f]),$$
%     where $f$ is a $\cf$-measurable function and $\Vert f\Vert_\infty=\sup_{x\in\cx}\vert f(x)\vert$. 
% \end{proposition}

% \begin{remark}\label{dualremark}
% If $\cx$ is a Polish space and $\cb(\cx)$ is the Borel $\sigma$-algebra, restricting $f$ to be a continuous function with $\Vert f\Vert_\infty\leq 1$ gives the same supremum in Proposition \ref{dual-v}. 
% \end{remark}

Now we generalize the above notions from two probability measures to the context of two sets of probability measures. Let $\mathcal{P}(\cx)$ denote the collection of all sets of probability measures on the measurable space $(\cx,\cf)$. We give the definitions of generalized KL and JS divergences as follows.

\begin{definition}\label{KL}
    Let $\cp,\cq\in\cp(\cx)$, we define the generalized KL divergence between $\cp$ and $\cq$ in the following form:
    \begin{align}\label{KL}
        \textup{KL}(\cp\parallel\cq)=\sup_{\mu\in\cp}\inf_{\nu\in\cq}\textup{KL}(\mu\parallel\nu).
    \end{align}
    %where $\textup{KL}(\mu\parallel\nu)$ is the classical KL divergence in Definition \ref{CKL}.
\end{definition}

% \begin{remark}
% We consider the generalized KL divergence to be a quantity that measures the difference between two sets of probability measures, and it is natural to expect the corresponding generalized KL divergence to be as small as possible when the difference between the two sets of probability measures is relatively small. It is clear that 
% $$\normalfont\inf_{\nu\in\cq}\sup_{\mu\in\cp}\text{KL}(\mu\parallel \nu)\geq \sup_{\mu\in\cp}\inf_{\nu\in\cq}\text{KL}(\mu\parallel\nu).$$
% Thus, we prefer the $\sup\inf$ operator in Definition \ref{KL}. 
% \end{remark}

\begin{remark}
 It is clear that KL divergence is not symmetric in general, i.e.,
 \normalfont
 $$\text{KL}(\cp\parallel\cq)\neq \text{KL}(\cq\parallel\cp).$$ 
\end{remark}
\begin{remark}\label{remark3}
\cite{yang2018robust} proposed the robust KL divergence between two sets $\cp$ and $\cq$ defined as $\normalfont{\text{KL}_*}(\cp\parallel\cq):=\inf_{\mu\in\cp,\nu\in\cq}\text{KL}(\mu\parallel\nu)$,
and focused on the case of $\normalfont{\text{KL}_*}(\{\mu\}\parallel\cq)$ to study the universal hypothesis testing for continuous distributions. In this case, $\normalfont\text{KL}(\{\mu\}\parallel\cq)={\text{KL}_*}(\{\mu\}\parallel\cq)$.  
\end{remark}

% The following example shows that the generalized KL divergence can be used to characterize distributional uncertainty.
% \begin{example}
% For a nominal probability measure $\eta$ and $\rho\geq 0$, consider the following KL ball based on the classic KL divergence, which is commonly used in distributionally robust optimization (DRO), 
% $$\rho_\eta=\{\mu|\textup{KL}(\mu\parallel\eta)\leq \rho \}.$$
% We consider a special KL ball based on the generalized KL divergence, where the second variable is a single point set. Let $B_\eta$ be the largest set $\cp$ such that $\textup{KL}(\cp\parallel\{\eta\})=\rho$, i.e.,
% $$A_\eta=\{\cp|\textup{KL}(\cp\parallel\{\eta\})=\rho\}, \ B_\eta=\bigcup_{\cp\in A_\eta}\cp,$$ 
% then we have $\rho_\eta=B_\eta$. 

% In fact, by definition, we have $\textup{KL}(\rho_\eta\parallel\{\eta\})=\sup_{\mu\in\rho_\eta}\textup{KL}(\mu\parallel\eta)=\rho$, and thus $\rho_\eta\subset B_\eta$. For each $\mu\in B_\eta$, there exists $\cp\in A_\eta$ such that $\mu\in\cp$ and $\sup_{\nu\in\cp}\textup{KL}(\nu\parallel\eta)=\rho$, then we have $\textup{KL}(\mu\parallel\eta)\leq\rho$, that is, $\mu\in\rho_\eta$, thus $B_\eta\subset\rho_\eta$.   
% \end{example}

%Now we consider the generalized Jensen-Shannon divergence.

\begin{definition}
     Let $\cp,\cq\in\cp(\cx)$, we define the generalized JS divergence in the following form:    
$$\normalfont\text{JS}(\cp,\cq)=\frac{1}{2}\text{KL}(\cp\parallel \cm)+\frac{1}{2}\text{KL}(\cq\parallel \cm),$$
    where $\cm:=\frac{\cp+\cq}{2}$, i.e., $\cm=\lt\{\frac{\mu+\nu}{2}:\mu\in\cp,\nu\in\cq\rt\}$.
\end{definition}

\begin{remark}
JS divergence is symmetric, i.e., $\textup{JS}(\cp,\cq)=\textup{JS}(\cq,\cp)$. 
\end{remark}

We also define the generalized total variation metric between two sets of probability measures by a Hausdorff-type distance function. 
\begin{definition}
    Let $\cp,\cq\in\cp(\cx)$, we define the generalized total variation metric between $\cp$ and $\cq$ as follows: 
$$V(\cp,\cq)=\max\lt(\sup_{\mu\in\cp}\inf_{\nu\in\cq}V(\mu,\nu),    
\sup_{\nu\in\cq}\inf_{\mu\in\cp}V(\mu,\nu)\rt).$$
\end{definition}

\begin{remark}
We note that $V(\cp,\cq)$ is a semi-metric, and it will be a metric if we only   
consider weakly compact elements in $\cp(\cx)$ and $\cx$ is a Polish space.
\end{remark}

The notion of weak convergence for the sets of probability measures was introduced in \cite{peng2010tightness,peng2019} under the framework of sublinear expectations. For each $\mathcal{P}\in\mathcal{P}(\cx)$, we define the sublinear expectation $\be^\mathcal{P}[\cdot]$ on $L^1(\cx)$ as follows:
\begin{align}
\mathbb{E}^\mathcal{P}[X]:=\sup_{\mu\in\mathcal{P}}E_\mu[X], \ \forall X\in L^1(\cx),
\end{align}
where $L^1(\cx)$ is the space of all $\cf$-measurable functions with the finite mean.

% It is easy to check that $\mathbb{E}^\mathcal{P}[\cdot]$ satisfies the following properties: for $X,Y\in L^1(\cx)$,
% \begin{itemize}
%     \item[(i)] Monotonicity: $\mathbb{E}^\cp[X]\leq \mathbb{E}^\cp[Y], \ \forall X\leq Y;$
%     \item[(ii)] Constant preserving: $\mathbb{E}^\cp[c]=c,\ \forall c\in\br;$
%     \item[(iii)] Sub-additivity: $\mathbb{E}^\cp[X+Y]\leq\mathbb{E}^\cp[X]+\mathbb{E}^\cp[Y];$
%     \item[(iv)] Positive homogeneity: $\mathbb{E}^\cp[\lambda X]=\lambda\mathbb{E}^\cp[X],\ \forall\lambda\geq 0.$
% \end{itemize}
% Thus, $\mathbb{E}^\cp[\cdot]$ is a sublinear expectation on the space $(\cx,L^1(\cx))$.

\begin{definition}
A sequence $\{\mathbb{E}^{\cp_n}[\cdot]\}_{n=1}^\infty$ of sublinear expectations is said to converge weakly to  $\mathbb{E}^\cp[\cdot]$, if for each  $\vp\in C_b(\cx)$,
$$\lim_{n\to\infty}\mathbb{E}^{\cp_n}[\vp]=\mathbb{E}^\cp[\vp].$$
\end{definition}
In this case, we also say that $\{\cp_n\}_{n=1}^\infty$ weakly converges to $\cp$.

\cite{li2017generalized} introduced the generalized Wasserstein distance for two probability sets and demonstrated that weak convergence of sublinear expectations can be characterized by means of this distance. 

\begin{definition}
    Let $\cp,\cq\in\cp(\cx)$. For $p\geq 1$, we define the $p$-order Wasserstein metric between $\cp$ and $\cq$ in the following form:
    $$\mathcal{W}_p(\cp,\cq)=\max\lt(\sup_{\mu\in\cp}\inf_{\nu\in\cq}W_p(\mu,\nu),\sup_{\nu\in\cq}\inf_{\mu\in\cp}W_p(\mu,\nu)\rt),$$
    where $W_p(\mu,\nu)$ is the $p$-order Wasserstein distance between $\mu$ and $\nu$, i.e., 
    $$W_p(\mu,\nu)=\inf\lt\{E[d(X,Y)^p]^{\frac{1}{p}}|\text{law}(X)=\mu,\text{law}(Y)=\nu\rt\}.$$
\end{definition}

In the remainder of this paper, we always assume that $\cx$ is a Polish space with the metric $d$, $\cb(\cx)$ is a Borel $\sigma$-algebra on it. We restrict our attention to the case where the set $\cp$ of probability measures is weakly compact and convex. Furthermore, we denote by $\overline{\cp}(\cx)$ the collection of all these weakly compact and convex sets $\cp$.

\section{Properties of the generalized divergences}
We first generalize the duality formula for KL divergence (Proposition \ref{dual}) to the generalized KL divergence.
\begin{theorem}[Duality formula]
    %\normalfont\text{(Duality formula for the generalized KL-divergence)}
    Let $\cp,\cq\in\overline{\cp}(\cx)$, we have
    \normalfont
    \begin{align*}
        \text{KL}(\cp\parallel \cq)=\sup_{\vp\in C_b(\cx)}\lt\{\mathbb{E}^{\cp}[\vp]-\text{log}\lt(\mathbb{E}^{\cq}[e^\vp]\rt)\rt\}.
    \end{align*}
\end{theorem}

\begin{proof}
    For fixed $\mu^*\in{\cp}$, we define a function $f(\vp,\nu)$ on $C_b(\cx)\times\cq$:    
    $$f(\vp,\nu)=E_{\mu^*}[\vp]-\text{log}\lt(E_\nu[e^\vp]\rt).$$ 
   We can verify that $f$ is convex on $\nu$ and concave on $\vp$.
    
    Indeed, for any $\nu_1,\nu_2\in\cq$ and $\lambda_1,\lambda_2>0$ with $\lambda_1+\lambda_2=1$, we have    
        \begin{align*}        
            f(\vp,\lambda_1\nu_1+\lambda_2\nu_2)        
            &=E_{\mu^*}[\vp]-\text{log}\lt(E_{\lambda_1\nu_1+\lambda_2\nu_2}[e^\vp]\rt) \\        
            &=E_{\mu^*}[\vp]-\text{log}\lt(\lambda_1E_{\nu_1}[e^\vp]+\lambda_2E_{\nu_2}[e^\vp]\rt) \\        
            &\leq E_{\mu^*}[\vp]-\lambda_1\text{log}\lt(E_{\nu_1}[e^\vp]\rt)-\lambda_2\text{log}\lt(E_{\nu_2}[e^\vp]\rt) \\        
            %&=\lambda_1(E_{\mu^*}[\vp]-\text{log}\lt(E_{\nu_1}[e^\vp]\rt)+\lambda_2(E_{\mu^*}[\vp]-\text{log}\lt(E_{\nu_2}[e^\vp]\rt) \\        
            &=\lambda_1f(\vp,\nu_1)+\lambda_2f(\vp,\nu_2).   
        \end{align*}
        
   For the concavity of $f$ on $\vp$,  let $\vp_1,\vp_2\in C_b(\cx)$ and $\lambda_1,\lambda_2>0$ with $\lambda_1+\lambda_2=1$. By H\"{o}lder's inequality, 
   we have
        \begin{align*}        
            f(\lambda_1\vp_1+\lambda_2\vp_2,\nu)        
            &=E_{\mu^*}[\lambda_1\vp_1+\lambda_2\vp_2]-\text{log}\lt(E_\nu[e^{\lambda_1\vp_1+\lambda_2\vp_2}]\rt)\\   
            %&=\lambda_1E_{\mu^*}[\vp_1]+\lambda_2E_{\mu^*}[\vp_2]-\text{log}\lt(E_\nu[e^{\lambda_1\vp_1} e^{\lambda_2\vp_2}]\rt) \\        
            &\geq E_{\mu^*}[\lambda_1\vp_1+\lambda_2\vp_2]-\text{log}\lt\{\lt(E_\nu[e^{\vp_1}]\rt)^{\lambda_1}\lt(E_\nu[e^{\vp_2}]\rt)^{\lambda_2}\rt\} \\ 
            %&=\lambda_1E_{\mu^*}[\vp_1]+\lambda_2E_{\mu^*}[\vp_2]-\text{log}\lt\{\lt(E_\nu[e^{\vp_1}]\rt)^{\lambda_1}\lt(E_\nu[e^{\vp_2}]\rt)^{\lambda_2}\rt\} \\        
            &=\lambda_1E_{\mu^*}[\vp_1]+\lambda_2E_{\mu^*}[\vp_2]-\lambda_1\text{log}\lt(E_\nu[e^{\vp_1}]\rt)-\lambda_2\text{log}\lt(E_\nu[e^{\vp_2}]\rt) \\        
            &=\lambda_1f(\vp_1,\nu)+\lambda_2f(\vp_2,\nu),    
        \end{align*}
    which implies that $f(\vp,\nu)$ is concave on $\vp$. 
    
    Thanks to the minimax theorem (see Corollary 3.3 in \cite{sion1958general}) and Proposition \ref{dual}, we immediately obtain 
    \begin{align*}
        \text{KL}(\cp\parallel\cq)&=\sup_{\mu\in\cp}\inf_{\nu\in\cq}\text{KL}(\mu\parallel \nu)\\
        &=\sup_{\mu\in\cp}\inf_{\nu\in\cq}\sup_{\vp\in C_b(\cx)}\lt\{E_\mu[\vp]-\text{log}\lt(E_\nu[e^\vp]\rt)\rt\} \\
        &=\sup_{\mu\in\cp}\sup_{\vp\in C_b(\cx)}\inf_{\nu\in\cq}\lt\{E_\mu[\vp]-\text{log}\lt(E_\nu[e^\vp]\rt)\rt\} \\
        &=\sup_{\vp\in C_b(\cx)}\sup_{\mu\in\cp}\inf_{\nu\in\cq}\lt\{E_\mu[\vp]-\text{log}\lt(E_\nu[e^\vp]\rt)\rt\} \\
        &=\sup_{\vp\in C_b(\cx)}\lt\{\mathbb{E}^{\cp}[\vp]-\text{log}\lt(\mathbb{E}^{\cq}[e^\vp]\rt)\rt\}.
    \end{align*}
\end{proof}
The generalized KL divergence and JS divergence admit the well properties as in the classical case.

\begin{proposition}\label{KL-p}
Let $\cp,\cq\in\overline{\cp}(\cx)$, we have 
    \begin{itemize}        
        \item[(i)] $\textup{KL}(\cp\parallel\cq)\geq 0;$        
        \item[(ii)] $\textup{KL}(\cp\parallel\cq)=0\iff\cp\subset\cq;$
        \item[(iii)] $\textup{KL}(\cp\parallel \cq)=0$ and $\textup{KL}(\cq\parallel \cp)=0$ $\iff$ $\cp=\cq;$
        \item[(iv)] $\textup{KL}(\cp\parallel\cq_1)\geq\textup{KL}(\cp\parallel\cq_2),$ if $\cq_1\subseteq\cq_2$;      \\                $\textup{KL}(\cp_2\parallel\cq)\geq\textup{KL}(\cp_1\parallel\cq),$ if $\cp_1\subseteq\cp_2$;
        %\item[(iv)] Fix $\cp\in\overline{\cp}(\cx)$, $\text{KL}(\cdot\parallel\cp)$ is a convex function;
        %\item[(v)] Fix $\cp\in\overline{\cp}(\cx)$, $\text{KL}(\cp\parallel\cdot)$ is a convex function;
        \item[(v)] $\textup{KL}(\cdot\parallel\cdot)$ is a convex function. 
    \end{itemize}
\end{proposition}

\begin{proof}
It is obvious that $(i)$ holds since $\textup{KL}(\mu\parallel\nu)\geq 0$. 

For $(ii)$, if $\text{KL}(\cp\parallel\cq)=0$, then for each $\mu\in\cp$, we have $\inf_{\nu\in\cq}\text{KL}(\mu\parallel\nu)=0$. Since $\cq$ is weakly compact and KL divergence is a lower semi-continuous function, there exists a $\nu\in\cq$ such that $\text{KL}(\mu\parallel\nu)=0$, thus $\mu=\nu$, which implies that $\cp\subset\cq$. On the other hand, if $\cp\subset\cq$, then
\begin{align*}
    0\leq\text{KL}(\cp\parallel\cq)&=\sup_{\vp\in C_b(\cx)}\lt\{\mathbb{E}^{\cp}[\vp]-\text{log}\lt(\mathbb{E}^{\cq}[e^\vp]\rt)\rt\}\\
    &\leq\sup_{\vp\in C_b(\cx)}\lt\{\mathbb{E}^{\cq}[\vp]-\text{log}\lt(\mathbb{E}^{\cq}[e^\vp]\rt)\rt\}\leq0.
\end{align*}

Obviously, $(iii)$ and $(iv)$ hold.  For $(v)$, let $\cp_1,\cp_2,\cq_1,\cq_2\in\overline{\cp}(\cx)$ and $\lambda_1,\lambda_2\in[0,1]$ with $\lambda_1+\lambda_2=1$, we have 
     \begin{align*}                
        &~~~\text{KL}(\lambda_1\cp_1+\lambda_2\cp_2\parallel\lambda_1\cq_1+\lambda_2\cq_2)    \\          
        =&\sup_{\vp\in  C_b(\cx)}\lt\{\be^{\lambda_1\cp_1+\lambda_2\cp_2}[\vp]-\text{log}\lt(\be^{\lambda_1\cq_1+\lambda_2\cq_2}[e^\vp]\rt)\rt\} \\   
        \leq&\sup_{\vp\in C_b(\cx)}\lt\{\lambda_1\be^{\cp_1}[\vp]+\lambda_2\be^{\cp_2}[\vp]-\text{log}\lt(\sup_{\mu_1\in\cq_1,\mu_2\in\cq_2}\lt\{\lambda_1 E_{\mu_1}[e^\vp]+\lambda_2E_{\mu_2}[e^\vp]\rt\}\rt)\rt\} \\
        =&\sup_{\vp\in  C_b(\cx)}\lt\{\lambda_1\be^{\cp_1}[\vp]+\lambda_2\be^{\cp_2}[\vp]-\sup_{\mu_1\in\cq_1,\mu_2\in\cq_2}\lt\{\text{log}\lt(\lambda_1 E_{\mu_1}[e^\vp]+\lambda_2E_{\mu_2}[e^\vp]\rt)\rt\}\rt\} \\
        \leq&\sup_{\vp\in  C_b(\cx)}\left\{\lambda_1\be^{\cp_1}[\vp]+\lambda_2\be^{\cp_2}[\vp]-\lambda_1\text{log}\lt(\be^{\cq_1}[e^\vp]\rt)-\lambda_2\text{log}\lt(\be^{\cq_2}[e^\vp]\rt)\right\} \\ 
        \leq&\lambda_1\text{KL}(\cp_1\parallel\cq_1)+\lambda_2\text{KL}(\cp_2\parallel\cq_2),    
    \end{align*}
where $\lambda_1\cp_1+\lambda_2\cp_2:=\{\lambda_1\mu_1+\lambda_2\mu_2|\mu_1\in\cp_1,\mu_2\in\cp_2\}$ and 
$\lambda_1\cq_1+\lambda_2\cq_2$ is defined similarly.
%:=\{\lambda_1\nu_1+\lambda_2\nu_2|\nu_1\in\cq_1,\nu_2\in\cq_2\}$. 
\end{proof}

\begin{remark}
By (ii), if $\cp=\{\mu\}$ and $\cq=\{\nu\}$ are singleton sets, then $\textup{KL}(\cp\parallel\cq)=0$ implies that $\mu\in\cq=\{\nu\}$, thus $\mu=\nu$, which is consistent with the classical results. We also note that (ii) and (iii) do not hold for ${\textup{KL}_*}(\cp\parallel\cq)$ in Remark \ref{remark3}.  
\end{remark}

\begin{corollary}\label{KL-c}
    Fix $\cp\in\overline{\cp}(\cx)$, $\normalfont\text{KL}(\cdot,\cp)$ and $\normalfont\text{KL}(\cp,\cdot)$ are convex functions.
\end{corollary}
% \begin{proof}
%     For $\cq_1,\cq_2\in\overline{\cp}(\cx)$ and $\lambda_1,\lambda_2\in [0,1]$ with $\lambda_1+\lambda_2=1$, 
%     since $\cp$ is a convex set, it is easy to check that $\lambda_1\cp+\lambda_2\cp=\cp$, then we have  
%     \begin{align*}
%         \text{KL}(\lambda_1\cq_1+\lambda_2\cq_2\parallel\cp)&=\text{KL}(\lambda_1\cq_1+\lambda_2\cq_2\parallel\lambda_1\cp+\lambda_2\cp) \\
%         &\leq\lambda_1\text{KL}(\cq_1\parallel\cp)+\lambda_2\text{KL}(\cq_2\parallel\cp), 
%     \end{align*}
%     which indicates that $\text{KL}(\cdot,\cp)$ is a convex function. Similarly, $\text{KL}(\cp,\cdot)$ is convex.
% \end{proof}

\begin{corollary}

    For $\cp,\cq\in\overline{\cp}(\cx)$, we have
    \begin{itemize}
        \item[(i)] $0\leq\textup{JS}(\cp\parallel\cq)\leq \textup{log}2;$
        \item[(ii)] $\textup{JS}(\cp\parallel\cq)=0\iff\cp=\cq;$
        %\item[(iii)] Fix $\cp\in\overline{\cp}(\cx)$, $\text{JS}(\cdot,\cp)$ and $\text{JS}(\cp,\cdot)$ are convex functions; 
        \item[(iii)] $\textup{JS}(\cdot,\cdot)$ is a convex function. 
    \end{itemize}
\end{corollary}

\section{Convergent Results}
We first generalize Pinsker's inequality (Proposition \ref{pinsker}) to the setting of sets of probability measures, which is used to obtain convergent results. 

\begin{theorem}[The generalized Pinsker's inequality]
    %\normalfont\text{(The generalized Pinsker's inequality)}
    Let $\cp,\cq\in\overline{\cp}(\cx)$, then
    \normalfont
    \begin{equation*}
        \max\{\text{KL}(\cp\parallel\cq),\text{KL}(\cq\parallel\cp)\}\geq 2 V^2(\cp,\cq).
    \end{equation*}
\end{theorem}

\begin{proof}
    By the Proposition \ref{pinsker}, we have
    \begin{align*}
        \text{KL}(\cp\parallel\cq)
        &=\sup_{\mu\in\cp}\inf_{\nu\in\cq}\text{KL}(\mu\parallel\nu) \\
        &\geq 2\sup_{\mu\in\cp}\inf_{\nu\in\cq}V^2(\mu,\nu)=2(\sup_{\mu\in\cp}\inf_{\nu\in\cq}V(\mu,\nu))^2.
    \end{align*}
    Similarly,
    $\text{KL}(\cq\parallel\cp)\geq 2(\sup_{\nu\in\cq}\inf_{\mu\in\cp}V(\mu,\nu))^2$. Thus we obtain
    
    \[\max\{\text{KL}(\cp\parallel\cq),\text{KL}(\cq\parallel\cp)\}\geq 2 V^2(\cp,\cq).\]
\end{proof}

For $\cp,\cq\in\overline{\cp}(\cx)$, let 
$\overline{\text{KL}}(\cp,\cq):=\text{KL}(\cp\parallel\cq)+\text{KL}(\cq\parallel\cp),$    
we have the following theorems. 

\begin{theorem}
   Let $\cp, \cp_n\in\overline{\cp}(\cx)$ for $n\geq 1$.
    If $\lim_{n\to\infty}\normalfont \overline{\text{KL}}(\cp_n,\cp)=0$, then 
   \begin{itemize}
       \item[(i)] $\lim_{n\to\infty}V(\cp_n,\cp)=0.$ 
       \item[(ii)] $\lim_{n\to\infty}\normalfont\text{JS}(\cp_n,\cp)=0.$
   \end{itemize} 
\end{theorem}

\begin{proof}
For $(i)$,  by the generalized Pinsker's inequality, it is easy to check that
    \begin{align*}
        V(\cp_n,\cp)\leq \sqrt{\frac{1}{2}\max\lt\{\text{KL}(\cp_n\parallel\cp),\text{KL}(\cp\parallel\cp_n)\rt\}}\leq\sqrt{\frac{1}{2}\overline{\text{KL}}(\cp_n,\cp)}, 
    \end{align*}
    thus $\lim_{n\to\infty}V(\cp_n,\cp)=0.$

For $(ii)$, by Corollary \ref{KL-c}, we have 
    \begin{align*}
        0&\leq\text{JS}(\cp_n,\cp) \\
        &=\frac{1}{2}\text{KL}\lt(\cp_n\parallel\frac{\cp_n+\cp}{2}\rt)+\frac{1}{2}\text{KL}\lt(\cp\parallel\frac{\cp_n+\cp}{2}\rt) \\
        &\leq \frac{1}{2}\lt(\frac{1}{2}\text{KL}(\cp_n\parallel\cp_n)+\frac{1}{2}\text{KL}(\cp_n\parallel\cp)\rt)+\frac{1}{2}\lt(\frac{1}{2}\text{KL}(\cp\parallel\cp_n)+\frac{1}{2}\text{KL}(\cp\parallel\cp)\rt) \\
        &=\frac{1}{4}\lt(\text{KL}(\cp_n\parallel\cp)+\text{KL}(\cp\parallel\cp_n)\rt)=\frac{1}{4}\overline{\text{KL}}(\cp_n,\cp).
    \end{align*}
    Let $n\to\infty$, we have $\lim_{n\to\infty}\text{JS}(\cp_n,\cp)=0$.
\end{proof}

\begin{theorem}
    Let $\cp, \cp_n\in\overline{\cp}(\cx)$ for $n\geq 1$.
    Then $\lim_{n\to\infty}\normalfont\text{JS}(\cp_n,\cp)=0$, if and only if $\lim_{n\to\infty}V(\cp_n,\cp)=0$.
\end{theorem}

\begin{proof} We first show that
$$\lim_{n\to\infty}\normalfont\text{JS}(\cp_n,\cp)=0 \Rightarrow \lim_{n\to\infty}V(\cp_n,\cp)=0.$$
    If it does not hold, then there exist $\epsilon>0$ and a subsequence $\{\cp_{n_k}\}$ such that $V(\cp_{n_k},\cp)>\epsilon$ holds for each $k$. Without loss of generality, we assume that there exists a subsequence $\{\cp_{n_k}\}$ such that $\sup_{\mu_k\in\cp_{n_k}}\inf_{\nu\in\cp}V(\mu_k,\nu)>\epsilon$ for each $k$. Therefore, there exists a $\mu_k\in\cp_{n_k}$ such that $\inf_{\nu\in\cp}V(\mu_k,\nu)>\epsilon$, since $\cp$ is weakly compact and total variation distance is a lower semi-continuous function, there exists a $\nu_k\in\cp$ such that $V(\mu_k,\nu_k)>\epsilon$. Let $\cm_{n_k}=\frac{\cp_{n_k}+\cp}{2}$, then $\forall \ \gamma\in\cm_{n_k}$, we have 
    $$V(\mu_k,\gamma)+V(\nu_k,\gamma)\geq V(\mu_k,\nu_k)>\epsilon,$$
    thus $V(\mu_k,\gamma)>\frac{\epsilon}{2}$ or $V(\nu_k,\gamma)>\frac{\epsilon}{2}$. 
    
    By Pinsker's inequality, we have
    $$\text{KL}(\mu_k\parallel\gamma)>\frac{\epsilon^2}{2}\ \text{or} \ \text{KL}(\nu_k\parallel\gamma)>\frac{\epsilon^2}{2}.$$ 
    Thus
    $$\inf_{\gamma\in\cm_{n_k}}\text{KL}(\mu_k\parallel\gamma)\geq\frac{\epsilon^2}{2}\ \text{or} \
    \inf_{\gamma\in\cm_{n_k}}\text{KL}(\nu_k\parallel\gamma)\geq\frac{\epsilon^2}{2}.$$
    Furthermore, 
    $$\sup_{\mu_k\in\cp_{n_k}}\inf_{\gamma\in\cm_{n_k}}\text{KL}(\mu_k\parallel\gamma)\geq\frac{\epsilon^2}{2}\ \text{or} \
    \sup_{\nu\in\cp}\inf_{\gamma\in\cm_{n_k}}\text{KL}(\nu\parallel\gamma)\geq\frac{\epsilon^2}{2},$$
    which implies that $\text{KL}(\cp_{n_k}\parallel\cm_{n_k})\geq\frac{\epsilon^2}{2}\ \text{or}\ \text{KL}(\cp\parallel\cm_{n_k})\geq\frac{\epsilon^2}{2}$. 
    
    We note that 
    $$\text{JS}(\cp_{n_k},\cp)=\frac{1}{2}\text{KL}(\cp_{n_k}\parallel\cm_{n_k})+\frac{1}{2}\text{KL}(\cp\parallel\cm_{n_k})\geq\frac{\epsilon^2}{4}>0,$$
    which contradicts $\lim_{n_k\to\infty}\text{JS}(\cp_{n_k},\cp)=0$.  

Secondly, if $\lim_{n\to\infty}V(\cp_n,\cp)=0$, then 
$\lim_{n\to\infty}\sup_{\mu_n\in\cp_n}\inf_{\nu\in\cp}V(\mu_n,\nu)=0$ and $\lim_{n\to\infty}\sup_{\nu\in\cp}\inf_{\mu_n\in\cp_n}V(\mu_n,\nu)=0$. Thus, $\forall\epsilon>0,\exists\ N\in\bn$, when $n\geq N$, we have 
$$\sup_{\mu_n\in\cp_n}\inf_{\nu\in\cp}V(\mu_n,\nu)<\epsilon\  \text{and} \ \sup_{\nu\in\cp}\inf_{\mu_n\in\cp_n}V(\mu_n,\nu)<\epsilon.$$ 
Then if $n\geq N$, for each $\mu_n\in\cp_n$, there exists a $\nu_n\in\cp$ such that $V(\mu_n,\nu_n)<\epsilon$, let $\cm_{n}=\frac{\cp_n+\cp}{2}$, take $\lambda_n=\frac{\mu_n+\nu_n}{2}\in\cm_n$, we have $\text{KL}(\mu_n\parallel\lambda_n)< \epsilon$ (see (3.10) in \cite{lin1991divergence}), then $\inf_{\lambda\in\cm_n}\text{KL}(\mu_n\parallel\lambda)< \epsilon$, thus 
$$\sup_{\mu_n\in\cp_n}\inf_{\lambda\in\cm_n}\text{KL}(\mu_n\parallel\lambda)<\epsilon.$$ 
Similarly, we get $\sup_{\nu\in\cp}\inf_{\lambda\in\cm_n}\text{KL}(\nu\parallel\lambda)<\epsilon$, then we have $\text{JS}(\cp_n,\cp)<\epsilon$, which implies that $\lim_{n\to\infty}\text{JS}(\cp_n,\cp)=0$.
\end{proof}

Next, we generalize the duality formula for the total variation distance in Proposition \ref{dual} to the setting of sets of probability measures.

\begin{theorem}[Duality formula]
    %\normalfont\text{(Duality formula for the generalized total variation distance)}    
        Let $\cp,\cq\in\overline{\cp}(\cx)$, we have    
        $$V(\cp,\cq)=\frac{1}{2}\sup_{\vp\in\mathcal{C}_b(\cx),\Vert \vp\Vert_\infty\leq 1}\lt\vert \be^{\cp}[\vp]-\be^{\cq}[\vp]\rt\vert.$$
        %where $f\in\mathcal{C}_b(\cx)$ and $\Vert f\Vert_\infty=\sup_{x\in\cx}\vert f(x)\vert$.
\end{theorem}

\begin{proof}
    By the minimax theorem (see Corollary 3.3 in \cite{sion1958general}), we have 
    \begin{align*}
        \sup_{\mu\in\cp}\inf_{\nu\in\cq}V(\mu,\nu)&=\frac{1}{2}\sup_{\mu\in\cp}\inf_{\nu\in\cq}\sup_{\vp\in\mathcal{C}_b(\cx),\Vert \vp\Vert_\infty\leq 1}(E_\mu[\vp]-E_\nu[\vp]) \\
        &=\frac{1}{2}\sup_{\mu\in\cp}\sup_{\vp\in\mathcal{C}_b(\cx),\Vert \vp\Vert_\infty\leq 1}\inf_{\nu\in\cq}(E_\mu[\vp]-E_\nu[\vp]) \\
        &=\frac{1}{2}\sup_{\vp\in\mathcal{C}_b(\cx),\Vert \vp\Vert_\infty\leq 1}(\be^{\cp}[\vp]-\be^{\cq}[\vp]).
    \end{align*}
    Similarly, 
    $\sup_{\nu\in\cq}\inf_{\mu\in\cp}V(\mu,\nu)=\frac{1}{2}\sup_{\vp\in\mathcal{C}_b(\cx),\Vert \vp\Vert_\infty\leq 1}(\be^{\cq}[\vp]-\be^{\cp}[\vp]).$
    Thus, 
    \begin{align*}
        V(\cp,\cq)&=\max\lt(\sup_{\mu\in\cp}\inf_{\nu\in\cq}V(\mu,\nu),    
        \sup_{\nu\in\cq}\inf_{\mu\in\cp}V(\mu,\nu)\rt) \\
        &=\frac{1}{2}\sup_{\vp\in\mathcal{C}_b(\cx),\Vert \vp\Vert_\infty\leq 1}\lt\vert \be^{\cp}[\vp]-\be^{\cq}[\vp]\rt\vert.  
    \end{align*}
\end{proof}

\begin{theorem}
    Let $\cp, \cp_n\in\overline{\cp}(\cx)$ for $n\geq 1$.
    If $\lim_{n\to\infty}V(\cp_n,\cp)=0$, then $\be^{\cp_n}[\cdot]$ converges weakly to $\be^{\cp}[\cdot]$.
\end{theorem}

\begin{proof}
    For each $\vp\in\mathcal{C}_b(\cx)$, let $M_\vp=\sup_{x\in\cx}|\vp(x)|<+\infty$, we have 
    \begin{align*}
        \lt\vert \be^{\cp_n}[\vp]-\be^\cp[\vp]\rt\vert&=M_\vp\lt\vert \be^{\cp_n}[\vp/M_\vp]-\be^\cp[\vp/M_\vp]\rt\vert\leq 2M_\vp V(\cp_n,\cp).
    \end{align*}
    Let $n\to\infty$, we can deduce that $\lim_{n\to\infty}\be^{\cp_n}[\vp]=\be^\cp[\vp].$
   
\end{proof}

\begin{theorem}
    Let $\cp, \cp_n\in\overline{\cp}(\cx)$ for $n\geq 1$. Let $D=\sup_{x,y\in\cx}d(x,y)$, if $\lim_{n\to\infty}V(\cp_n,\cp)=0$ and $D<+\infty$, then $\lim_{n\to\infty}\mathcal{W}_1(\cp_n,\cp)=0.$
\end{theorem}

\begin{proof}
   For two probability measures $\mu$ and $\nu$ on a metric space $(\cx,\cb(\cx))$, by Particular Case 6.16 in \cite{villani2008optimal},  we have 
   $$W_1(\mu,\nu)\leq D\cdot V(\mu,\nu).$$ 
   According to the definition of the generalized total variation metric and the Wasserstein metric, $\lim_{n\to\infty}\mathcal{W}_1(\cp_n,\cp)=0$ holds.
\end{proof}

\begin{corollary}
    Let $\cp, \cp_n\in\overline{\cp}(\cx)$ for $n\geq 1$ and $\lim_{n\to\infty}\normalfont \overline{\text{KL}}(\cp_n,\cp)=0$. Then we have
    \begin{itemize}
        \item [(i)] $\be^{\cp_n}[\cdot]$ converges weakly to $\be^{\cp}[\cdot]$.
        \item [(ii)] If $D=\sup_{x,y\in\cx}d(x,y)<+\infty$, then $\lim_{n\to\infty}\mathcal{W}_1(\cp_n,\cp)=0$.
    \end{itemize}
\end{corollary}




%\begin{thebibliography}{00}

\bibliographystyle{elsarticle-harv} 
\bibliography{ref}

\begin{thebibliography}{17}
\expandafter\ifx\csname natexlab\endcsname\relax\def\natexlab#1{#1}\fi
\providecommand{\url}[1]{\texttt{#1}}
\providecommand{\href}[2]{#2}
\providecommand{\path}[1]{#1}
\providecommand{\DOIprefix}{doi:}
\providecommand{\ArXivprefix}{arXiv:}
\providecommand{\URLprefix}{URL: }
\providecommand{\Pubmedprefix}{pmid:}
\providecommand{\doi}[1]{\href{http://dx.doi.org/#1}{\path{#1}}}
\providecommand{\Pubmed}[1]{\href{pmid:#1}{\path{#1}}}
\providecommand{\bibinfo}[2]{#2}
\ifx\xfnm\relax \def\xfnm[#1]{\unskip,\space#1}\fi
%Type = Article
\bibitem[{Blei et~al.(2017)Blei, Kucukelbir and
  McAuliffe}]{blei2017variational}
\bibinfo{author}{Blei, D.M.}, \bibinfo{author}{Kucukelbir, A.},
  \bibinfo{author}{McAuliffe, J.D.}, \bibinfo{year}{2017}.
\newblock \bibinfo{title}{Variational inference: a review for statisticians}.
\newblock \bibinfo{journal}{J. Amer. Statist. Assoc.} \bibinfo{volume}{112},
  \bibinfo{pages}{859--877}.
%Type = Book
\bibitem[{Boucheron et~al.(2013)Boucheron, Lugosi and Massart}]{boucheron2013}
\bibinfo{author}{Boucheron, S.}, \bibinfo{author}{Lugosi, G.},
  \bibinfo{author}{Massart, P.}, \bibinfo{year}{2013}.
\newblock \bibinfo{title}{Concentration inequalities}.
\newblock \bibinfo{publisher}{Oxford University Press, Oxford}.
%Type = Book
\bibitem[{Cover and Thomas(2006)}]{cover2006elements}
\bibinfo{author}{Cover, T.M.}, \bibinfo{author}{Thomas, J.A.},
  \bibinfo{year}{2006}.
\newblock \bibinfo{title}{Elements of information theory}.
\newblock \bibinfo{edition}{Second} ed., \bibinfo{publisher}{Wiley-Interscience
  [John Wiley \& Sons], Hoboken, NJ}.
%Type = Article
\bibitem[{Donsker and Varadhan(1983)}]{Donsker1983AsymptoticEO}
\bibinfo{author}{Donsker, M.D.}, \bibinfo{author}{Varadhan, S.R.S.},
  \bibinfo{year}{1983}.
\newblock \bibinfo{title}{Asymptotic evaluation of certain {M}arkov process
  expectations for large time. {IV}}.
\newblock \bibinfo{journal}{Comm. Pure Appl. Math.} \bibinfo{volume}{36},
  \bibinfo{pages}{183--212}.
%Type = Article
\bibitem[{Ji et~al.(2023)Ji, Peng and Yang}]{JPY}
\bibinfo{author}{Ji, X.}, \bibinfo{author}{Peng, S.}, \bibinfo{author}{Yang,
  S.}, \bibinfo{year}{2023}.
\newblock \bibinfo{title}{Imbalanced binary classification under distribution
  uncertainty}.
\newblock \bibinfo{journal}{Inf.Sci.} \bibinfo{volume}{621},
  \bibinfo{pages}{156--171}.
%Type = Article
\bibitem[{Kullback and Leibler(1951)}]{kullback1951information}
\bibinfo{author}{Kullback, S.}, \bibinfo{author}{Leibler, R.A.},
  \bibinfo{year}{1951}.
\newblock \bibinfo{title}{On information and sufficiency}.
\newblock \bibinfo{journal}{Ann. Math. Statistics} \bibinfo{volume}{22},
  \bibinfo{pages}{79--86}.
%Type = Article
\bibitem[{Li and Lin(2017)}]{li2017generalized}
\bibinfo{author}{Li, X.}, \bibinfo{author}{Lin, Y.}, \bibinfo{year}{2017}.
\newblock \bibinfo{title}{Generalized {W}asserstein distance and weak
  convergence of sublinear expectations}.
\newblock \bibinfo{journal}{J. Theoret. Probab.} \bibinfo{volume}{30},
  \bibinfo{pages}{581--593}.
%Type = Article
\bibitem[{Lin(1991)}]{lin1991divergence}
\bibinfo{author}{Lin, J.}, \bibinfo{year}{1991}.
\newblock \bibinfo{title}{Divergence measures based on the {S}hannon entropy}.
\newblock \bibinfo{journal}{IEEE Trans. Inform. Theory} \bibinfo{volume}{37},
  \bibinfo{pages}{145--151}.
%Type = Article
\bibitem[{Pei et~al.(2021)Pei, Wang, Xu and Yue}]{pei2021worst}
\bibinfo{author}{Pei, Z.}, \bibinfo{author}{Wang, X.}, \bibinfo{author}{Xu,
  Y.}, \bibinfo{author}{Yue, X.}, \bibinfo{year}{2021}.
\newblock \bibinfo{title}{A worst-case risk measure by {G}-{V}a{R}}.
\newblock \bibinfo{journal}{Acta Math. Appl. Sin. Engl. Ser.}
  \bibinfo{volume}{37}, \bibinfo{pages}{421--440}.
%Type = Article
\bibitem[{Peng(2010)}]{peng2010tightness}
\bibinfo{author}{Peng, S.}, \bibinfo{year}{2010}.
\newblock \bibinfo{title}{Tightness, weak compactness of nonlinear expectations
  and application to clt.}
  \bibinfo{note}{\url{https://arxiv.org/abs/1006.2541}}.
%Type = Book
\bibitem[{Peng(2019)}]{peng2019}
\bibinfo{author}{Peng, S.}, \bibinfo{year}{2019}.
\newblock \bibinfo{title}{Nonlinear expectations and stochastic calculus under
  uncertainty}.
\newblock \bibinfo{publisher}{Springer, Berlin}.
%Type = Article
\bibitem[{Peng et~al.(2023)Peng, Yang and Yao}]{PYY}
\bibinfo{author}{Peng, S.}, \bibinfo{author}{Yang, S.}, \bibinfo{author}{Yao,
  J.}, \bibinfo{year}{2023}.
\newblock \bibinfo{title}{Improving value-at-risk prediction under model
  uncertainty}.
\newblock \bibinfo{journal}{J.Financ.Econom} \bibinfo{volume}{21},
  \bibinfo{pages}{228--259}.
%Type = Article
\bibitem[{Schennach and Wilhelm(2017)}]{schennach2017simple}
\bibinfo{author}{Schennach, S.M.}, \bibinfo{author}{Wilhelm, D.},
  \bibinfo{year}{2017}.
\newblock \bibinfo{title}{A simple parametric model selection test}.
\newblock \bibinfo{journal}{J. Amer. Statist. Assoc.} \bibinfo{volume}{112},
  \bibinfo{pages}{1663--1674}.
%Type = Article
\bibitem[{Sion(1958)}]{sion1958general}
\bibinfo{author}{Sion, M.}, \bibinfo{year}{1958}.
\newblock \bibinfo{title}{On general minimax theorems}.
\newblock \bibinfo{journal}{Pacific J. Math.} \bibinfo{volume}{8},
  \bibinfo{pages}{171--176}.
%Type = Article
\bibitem[{Van~Erven and Harremo\"es(2014)}]{van2014renyi}
\bibinfo{author}{Van~Erven, T.}, \bibinfo{author}{Harremo\"es, P.},
  \bibinfo{year}{2014}.
\newblock \bibinfo{title}{R\'enyi divergence and {K}ullback-{L}eibler
  divergence}.
\newblock \bibinfo{journal}{IEEE Trans. Inform. Theory} \bibinfo{volume}{60},
  \bibinfo{pages}{3797--3820}.
%Type = Book
\bibitem[{Villani(2009)}]{villani2008optimal}
\bibinfo{author}{Villani, C.}, \bibinfo{year}{2009}.
\newblock \bibinfo{title}{Optimal transport}.
\newblock \bibinfo{publisher}{Springer-Verlag, Berlin}.
%Type = Article
\bibitem[{Yang and Chen(2019)}]{yang2018robust}
\bibinfo{author}{Yang, P.}, \bibinfo{author}{Chen, B.}, \bibinfo{year}{2019}.
\newblock \bibinfo{title}{Robust {K}ullback-{L}eibler divergence and universal
  hypothesis testing for continuous distributions}.
\newblock \bibinfo{journal}{IEEE Trans. Inform. Theory} \bibinfo{volume}{65},
  \bibinfo{pages}{2360--2373}.

\end{thebibliography}

%\end{thebibliography}
\end{document}